\documentclass[10pt,a4paper]{article}
\usepackage{amsmath}
\usepackage{amsthm}
\usepackage{amsfonts}
\usepackage{latexsym}

\newcounter{alphthm}
\newtheorem{prop}{Proposition}[section]
\newtheorem{thm}{Theorem}[section]
\newtheorem{lem}[thm]{Lemma}
\newtheorem{cor}{Corollary}[section]

\theoremstyle{definition}

\newtheorem{ex}{Example}

\newcommand{\be}{\begin{equation}}
\newcommand{\ee}{\end{equation}}
\newcommand{\ben}{\begin{enumerate}}
\newcommand{\een}{\end{enumerate}}
\newcommand{\pa}{{\partial}}

\newcommand{\g}{{\bf g}}
\newcommand{\pxi}{{\pa \over \pa x^i}}

\title{Generalized Isotropic Berwald Manifolds}
\author{A. Tayebi and E. Peyghan}

\begin{document}

\maketitle
%--------------------------------------------------------------------------------------------------------------------
\begin{abstract}
In this paper, we construct a new class of Finsler manifolds called generalized isotropic Berwald  manifolds  which is an extension of  the class of isotropic Berwald  manifolds. We prove that every generalized isotropic Berwald  manifold is a generalized Douglas-Weyl manifold. On a compact generalized isotropic Berwald manifold, we show that the notions of stretch  and Landsberg curvatures are equivalent. Then we prove that on these manifolds,  a Finsler metric  is  R-quadratic if and only if it is a stretch metric with vanishing  ${\bar E}$-curvature. Finally,  we determine the flag curvature of generalized isotropic Berwald manifold with scalar flag curvature.
\end{abstract}
%--------------------------------------------------------------------------------------------------------------------
\section{Introduction}
%--------------------------------------------------------------------------------------------------------------------
In Finsler geometry,  there are several important non-Riemannian  quantities. Let $(M, F)$ be a Finsler manifold. The second derivatives of ${1\over 2} F_x^2$ at $y\in T_xM_0$ is an inner product $\g_y$ on $T_xM$.  The third order derivatives of ${1\over 2} F_x^2$ at  $y\in T_xM_0$ is a symmetric trilinear forms ${\bf C}_y$ on $T_xM$. We call $\g_y$ and ${\bf C}_y$ the  fundamental form and  the Cartan torsion, respectively. The rate of change of the Cartan torsion along geodesics is the  Landsberg curvature  $ {\bf L}_y$ on $T_xM$ for any $y\in T_xM_0$.  $F$ is said to be  Landsbergian if ${\bf L}=0$.

For a Finsler metric $F=F(x,y)$, its geodesics curves are characterized by the system of differential equations $ \ddot c^i+2G^i(\dot c)=0$, where the local functions $G^i=G^i(x, y)$ are called the spray coefficients and given by following
\[
G^i=\frac{1}{4}g^{il}\Big\{\frac{\partial^2[F^2]}{\partial x^k \partial y^l}y^k-\frac{\partial[F^2]}{\partial x^l}\Big\},\ \ y\in T_xM.
\]
$F$ is  called a Berwald metric if $G^i = {1\over 2} \Gamma^i_{jk}(x)y^jy^k$ is quadratic in $y\in T_xM$  for any $x\in M$ \cite{Be}. In \cite{Ich}, it is  proved that on  a Berwald space,  the parallel translation along any geodesic preserves the Minkowski  functionals. Then Berwald spaces can be viewed as  Finsler spaces modeled on a single Minkowski space.

Recently the various interesting special forms of Cartan, Berwald  and Landsberg tensors have been obtained by some Finslerians. The Finsler spaces having such special forms have been called C-reducible, P-reducible, semi-C-reducible, isotropic Berwald curvature, isotropic mean Berwald curvature, and  isotropic Landsberg curvature, etc  \cite{ChSh}\cite{Mat}\cite{PT}\cite{PTH}\cite{TN0}\cite{TP}.

%In \cite{M3}, Matsumoto  introduced the notion of C-reducible metrics and proved that any Randers metric is C-reducible. Then Matsumoto-H\={o}j\={o} prove that the converse is true \cite{MH}. A Randers metric $F=\alpha+\beta$ is just a Riemannian metric $\alpha$ perturbated by a one form $\beta$ which has important applications both in mathematics and physics \cite{TP1}. In \cite{Mat4},  Matsumoto-Shibata introduced the notion of semi-C-reducibility by considering the form of Cartan torsion of a non-Riemannian $(\alpha,\beta)$-metric on a manifold $M$ with dimension $n\geq 3$.

In  \cite{ChSh}, Shen-Chen by using the structure of Funk metric, introduce the notion of isotropic Berwald metrics. This motivates us to study special forms of Berwald  curvature for other important special Finsler metrics.

Let $(M, F)$ be a two-dimensional Finsler manifold. We refer to the Berwald's frame $(\ell^i, m^i)$ where $\ell^i=y^i/F(y)$,  $m^i$ is the unit vector with $\ell_i m^i=0$ and  $\ell_i=g_{ij}\ell^i$.  Then the Berwald curvature is given by
\[
B^i_{\ jkl}=F^{-1}(-2I_{,1}\ell^i+I_2m^i) m_j m_k m_l,
\]
where $I$ is 0-homogeneous function called the main scalar of $F$ and $I_2=I_{,2}+I_{,1|2}$ (see page 689 in \cite{An}). Since the Cartan tensor of $F$ is given by $C_{ijk}=F^{-1}I m_im_jm_k$, then  the Berwald curvature  can be written as folowing
\begin{equation}\label{3}
B^i_{\ jkl}=-\frac{2I_{,1}}{I}C_{jkl}\ell^i+\frac{I_2}{3F}\{h_{jk}h^i_l+h_{kl}h^i_j+h_{lj}h^i_k\},
\end{equation}
where $h_{ij}:=m_i m_j$ is the angular metric.

Let $(M, F)$ be a Finsler manifold. Then $F$ is said to  be {\it generalized isotropic Berwald metric} if its Berwald curvature satisfies  following
\begin{equation}\label{B0}
B^i_{\ jkl}=\mu C_{jkl}\ell^i+\lambda(h^i_j h_{kl}+h^i_k h_{jl}+h^i_l h_{jk}),
\end{equation}
where $\mu=\mu(x,y)$ and $\lambda=\lambda(x,y)$ are homogeneous functions of degrees 0 and -1 with respect to $y$, respectively. Then $(M, F)$ is called a generalized isotropic Berwald manifold. It is remarkable that, if $\mu=2c$ and $\lambda=cF^{-1}$, where $c=c(x)$ is a scalar function on $M$, then $F$ reduces to a isotropic Berwald metric \cite{TN2}.  Then the  class of generalized isotropic Berwald  manifolds  contains the class of  isotropic Berwald  manifolds, as a special case. By (\ref{B0}),  it results that every Finsler surface has generalized isotropic Berwald curvature with $\mu=\frac{-2I_{,1}}{I}$ and  $\lambda=\frac{I_2}{3}$.
\begin{ex}
Consider the   Funk metric on the unit ball $\mathbb{B}^n\subset\mathbb{R}^n$ defined by
\[
F(x, y):=\frac{\sqrt{|y|^2-(|x|^2|y|^2-<x,y>^2)}+<x,y>}{1-|x|^2},\ \ \
y\in T_x\mathbb{B}^n=\mathbb{R}^n
\]
where $|.|$ and $<,>$ denote the Euclidean norm and inner product in $\mathbb{R}^n$, respectively. $F$ is a generalized isotropic Berwald metric with  $\mu=1$ and  $\lambda=\frac{1}{2F}$.
\end{ex}

The Douglas tensor is another non-Riemanian curvature which  defined by
\[
D^i_{\ jkl}:=\big(G^i-\frac{1}{n+1}\frac{\partial G^m}{\partial y^m}y^i\big)_{y^j y^k y^l}.
\]
For more details see \cite{BM}\cite{NT}.  A Finsler metric is called a generalized Douglas-Weyl (GDW) metric if the Douglas tensor satisfy in $h^i_{\alpha} D^\alpha_{\ jkl|m}y^m=0$ \cite{NST1}. In \cite{BP}, B\'{a}cs\'{o}-Papp show that this class of Finsler metrics is closed under projective transformation. In this paper, we prove the following.
\begin{thm}\label{THM1}
Every  generalized isotropic Berwald metric is a generalized Douglas-Weyl metric.
\end{thm}

As a generalization of Landsberg curvature, L. Berwald introduced a non-Riemannian curvature so-called stretch curvature and denoted by ${\bf \Sigma}_y$ \cite{Be}. He  showed that this tensor vanishes if and only if the length of a vector remains unchanged under the parallel displacement along an infinitesimal parallelogram. Therefore,  we study  complete generalized isotropic Berwald manifold with vanishing stretch curvature and prove the following.

\begin{thm}\label{THM2}
Let $(M, F)$ be a complete generalized isotropic Berwald manifold and $\mu$ be bounded function  on $M$. Suppose that $F$ has vanishing stretch curvature. Then $F$ is a Landsberg metric.
\end{thm}

The second variation of geodesics gives rise to a family of linear maps $R_y: T_xM\rightarrow T_xM$,
at any point $y\in T_xM$, which is called the Riemann curvature in the direction $y$. A Finsler metric $F$ is said to be R-quadratic if the Riemannian curvature $R_y$ is quadratic in $y\in T_xM$ at each point $x\in M$.  In the sense of B$\acute{a}$cs$\acute{o}$-Matsumoto, $F$ is R-quadratic if and only if the h-curvature of Berwald connection depends on position only (\cite{BM3}\cite{TBN}\cite{TP}). Every Berwald metric and
R-flat metric is R-quadratic metric. On the other hand, in \cite{ShDiff}  Shen find a  new non-Riemannian quantity for Finsler metrics that  is closely related to the  ${E}$-curvature and call it ${\bar E}$-curvature. Recall that, the  ${\bar E}$-curvature is obtained from the mean Berwald curvature $E$, by the  horizontal covariant differentiation  along geodesics. In this paper, we study generalized isotropic Berwald  manifolds with R-quadratic metrics and prove the following.
\begin{thm}\label{THM3}
Let $(M, F)$ be a generalized isotropic Berwald manifold of dimension $n>2$. Then $F$ is  R-quadratic if and only if it is a stretch metric with \  ${\bf \bar E}=0$.
\end{thm}

For a Finsler manifold $(M, F)$, the flag curvature is  a function ${\bf K}(P, y)$ of tangent planes $P\subset T_xM$ and  directions $y\in P$. Indeed the flag curvature in Finsler geometry is a natural extension of the sectional curvature in Riemannian geometry. $F$  is said to be  of scalar flag curvature if the flag curvature ${\bf K}(P, y)={\bf K}(x, y)$ is independent of flags $P$ associated with any fixed flagpole $y$.  One of the important problems in Finsler geometry is to characterize Finsler manifolds of scalar flag curvature \cite{NT}. In this paper, we  study generalized isotropic Berwald metrics of scalar curvature  and   partially determine the flag curvature. More precisely, we prove the following.

\begin{thm}\label{THM4}
Let $(M, F)$ be an n-dimensional generalized isotropic Berwald manifold of scalar flag curvature. Then the flag curvature ${\bf K}={\bf K}(x, y)$  satisfies
\be
{n+1 \over 3} {\bf K}_{y^k} + \Big ({\bf K}+\frac{\mu^2}{4}-\frac{\mu'}{2F}\Big ) I_k=0,\label{KKc}
\ee
where $\mu':=\mu_{|s}y^s$ denotes the horizontal derivation of $\mu$ with respect to the Berwald connection.
\end{thm}

There are many connections in Finsler geometry   \cite{TAE}\cite{TN}. In this paper, we use the Berwald connection  and the $h$- and $v$- covariant derivatives of a Finsler tensor field are denoted by `` $|$ " and ``, " respectively.
%--------------------------------------------------------------------------------------------------------------------
\section{Preliminaries}
%--------------------------------------------------------------------------------------------------------------------
Let $M$ be a n-dimensional $ C^\infty$ manifold. Denote by $T_x M $ the tangent space at $x \in M$,  by $TM=\cup _{x \in M} T_x M $ the tangent bundle of $M$, and by $TM_{0} = TM \setminus \{ 0 \}$ the slit tangent bundle on $M$. A  Finsler metric on $M$ is a function $ F:TM \rightarrow [0,\infty)$ which has the following properties:\\
(i) $F$ is $C^\infty$ on $TM_{0}$;\\
(ii) $F$ is positively 1-homogeneous on the fibers of tangent bundle $TM$;\\
(iii) for each $y\in T_xM$, the following quadratic form ${\bf g}_y$ on
$T_xM$  is positive definite,
\[
{\bf g}_{y}(u,v):={1 \over 2} \frac{\partial^2}{\partial s \partial t}\left[  F^2 (y+su+tv)\right]|_{s,t=0}, \ \
u,v\in T_xM.
\]
Let  $x\in M$ and $F_x:=F|_{T_xM}$.  To measure the non-Euclidean feature of $F_x$, define ${\bf C}_y:T_xM\otimes T_xM\otimes T_xM\rightarrow \mathbb{R}$ by
\[
{\bf C}_{y}(u,v,w):={1 \over 2} \frac{d}{dt}\left[{\bf g}_{y+tw}(u,v)
\right]|_{t=0}, \ \ u,v,w\in T_xM.
\]
The family ${\bf C}:=\{{\bf C}_y\}_{y\in TM_0}$  is called the Cartan torsion. It is well known that ${\bf{C}}=0$ if and only if $F$ is Riemannian \cite{ShDiff}.

\bigskip

The horizontal covariant derivatives of ${\bf C}$ along geodesics give rise to  the  Landsberg curvature  ${\bf L}_y:T_xM\otimes T_xM\otimes T_xM\rightarrow \mathbb{R}$  defined by ${\bf L}_y(u,v,w):=L_{ijk}(y)u^iv^jw^k$, where
\[
L_{ijk}:=C_{ijk|s}y^s,
\]
$u=u^i{{\partial } \over {\partial x^i}}|_x$,  $v=v^i{{\partial }\over {\partial x^i}}|_x$ and $w=w^i{{\partial }\over {\partial x^i}}|_x$. The family ${\bf L}:=\{{\bf L}_y\}_{y\in TM_{0}}$  is called the Landsberg curvature. A Finsler metric is called a Landsberg metric  if {\bf{L}=0} \cite{TP1}. The quotient ${\bf L}/{\bf C}$ is regarded as the relative rate of change of  ${\bf C}$ along Finslerian geodesics.  A Finsler metric  is said to be relatively isotropic Landsberg metric if
\[
{\bf L}=\eta \bf C,
\]
where $\eta=\eta(x,y)$ is a homogeneous scalar function of degree 1  on $TM$.

\bigskip

Define the stretch curvature ${\bf \Sigma}_y:T_xM\otimes T_xM \otimes T_xM  \otimes T_xM\rightarrow \mathbb{R}$ by ${\bf \Sigma}_y(u, v, w,z):={\Sigma}_{ijkl}(y)u^iv^jw^kz^l$, where
\[
{\Sigma}_{ijkl}:=2(L_{ijk|l}-L_{ijl|k}).
\]
A Finsler metric  is said to be  stretch metric if ${\bf \Sigma}=0$. Every Landsberg metric is a  stretch metric.

\bigskip

Given a Finsler manifold $(M,F)$, then a global vector field ${\bf G}$ is induced by $F$ on $TM_0$, which in a standard coordinate $(x^i,y^i)$ for $TM_0$ is given by ${\bf G}=y^i {{\partial} \over {\partial x^i}}-2G^i(x,y){{\partial} \over {\partial y^i}}$, where
\[
G^i:=\frac{1}{4}g^{il}\{[F^2]_{x^ky^l}y^k-[F^2]_{x^l}\},\ \ y\in T_xM.
\]
The ${\bf G}$ is called the  spray associated  to $(M,F)$.  In local coordinates, a curve $c(t)$ is a geodesic if and only if its coordinates $(c^i(t))$ satisfy $ \ddot c^i+2G^i(\dot c)=0$.
\bigskip

For a tangent vector $y \in T_xM_0$, define ${\bf B}_y:T_xM\otimes T_xM \otimes T_xM\rightarrow T_xM$ and ${\bf E}_y:T_xM \otimes T_xM\rightarrow \mathbb{R}$ by ${\bf B}_y(u, v, w):=B^i_{\ jkl}(y)u^jv^kw^l{{\partial } \over {\partial x^i}}|_x$ and ${\bf E}_y(u,v):=E_{jk}(y)u^jv^k$
where
\[
B^i_{\ jkl}:={{\partial^3 G^i} \over {\partial y^j \partial y^k \partial y^l}},\ \ \ E_{jk}:={{1}\over{2}}B^m_{\ jkm}.
\]
The $\bf B$ and $\bf E$ are called the Berwald curvature and mean Berwald curvature, respectively.  Then $F$ is called a Berwald metric and weakly Berwald metric if $\bf{B}=0$ and $\bf{E}=0$, respectively  \cite{ShDiff}.

\bigskip

Define ${\bf D}_y:T_xM\otimes T_xM \otimes T_xM\rightarrow T_xM$  by
${\bf D}_y(u,v,w):=D^i_{\ jkl}(y)u^iv^jw^k\frac{\partial}{\partial x^i}|_{x}$ where
\[
D^i_{\ jkl}:=B^i_{\ jkl}-{2\over
n+1}\{E_{jk}\delta^i_l+E_{jl}\delta^i_k+E_{kl}\delta^i_j+E_{jk,l}
y^i\}.
\]
We call ${\bf D}:=\{{\bf D}_y\}_{y\in TM_{0}}$ the Douglas curvature. A Finsler metric with ${\bf D}=0$ is called a Douglas metric. The notion of Douglas metrics was proposed by B$\acute{a}$cs$\acute{o}$-Matsumoto as a generalization  of Berwald metrics \cite{BM}.

A Finsler metric is called a generalized Douglas-Weyl (GDW) metric if the Douglas tensor satisfy in
\[
h^i_{\alpha} D^\alpha_{\ jkl|m}y^m=0
\]
In \cite{BP}, B\'{a}cs\'{o}-Papp show that this class of Finsler metrics is closed under projective transformation. In \cite{NST1}, Najafi-Shen-Tayebi find the necessary and sufficient condition for a Randers metric to be a generalized Douglas-Weyl  metric.

\bigskip
The Riemann curvature ${\bf R}_y= R^i_{\ k}  dx^k \otimes \pxi|_x :
T_xM \to T_xM$ is a family of linear maps on tangent spaces, defined
by
\[
R^i_{\ k} = 2 {\pa G^i\over \pa x^k}-y^j{\pa^2 G^i\over \pa
x^j\pa y^k} +2G^j {\pa^2 G^i \over \pa y^j \pa y^k} - {\pa G^i \over
\pa y^j} {\pa G^j \over \pa y^k}.  \label{Riemann}
\]
The flag curvature in Finsler geometry is a natural extension of the sectional curvature in Riemannian geometry was  first introduced by L. Berwald \cite{Be}. For a flag $P={\rm span}\{y, u\} \subset T_xM$ with flagpole $y$, the  flag curvature ${\bf K}={\bf K}(P, y)$ is defined by
\begin{equation}\label{TP5}
{\bf K}(P, y):= {\g_y (u, {\bf R}_y(u)) \over \g_y(y, y) \g_y(u,u)
-\g_y(y, u)^2 }.
\end{equation}
When $F$ is Riemannian, ${\bf K}={\bf  K}(P)$ is independent of $y\in P$,  and is
the sectional curvature of $P$. We say that a Finsler metric $F$ is   of scalar curvature if for any $y\in T_xM$, the flag curvature ${\bf K}= {\bf K}(x, y)$ is a scalar function on the slit tangent bundle $TM_0$. If ${\bf K}=constant$, then $F$ is said to be of  constant flag curvature. A Finsler metric $F$ is called  isotropic flag curvature, if ${\bf K}= {\bf K}(x)$.

A Finsler metric $F$ is said to be R-quadratic if $R_y$ is quadratic in $y\in T_xM$ at each point $x\in M$. Let
\[
R^i_{\ jkl}(x,y):=\frac{1}{3}\frac{\partial }{\partial
y^j}\{\frac{\partial R^i_{\ k}}{\partial y^l}-\frac{\partial
R^i_{\ l}}{\partial y^k}\},
\]
where $R^i_{\ jkl}$ is the Riemann curvature of Berwald connection. Then we have $R^i_{\ k}=R^i_{\ jkl}(x,y)y^jy^l$. Therefore  $R^i_{\ k}$ is quadratic in $y\in T_xM$ if and only if $R^i_{\ jkl}$ are functions of position alone.  Indeed a Finsler metric is R-quadratic if and only if the h-curvature of Berwald connection depends on position only in the sense of B$\acute{a}$cs$\acute{o}$-Matsumoto \cite{BM3}.

%------------------------------------------------------------------------------------------------------------
\section{Proof of Theorem \ref{THM1}}
%------------------------------------------------------------------------------------------------------------
In this section, we are going to prove the  Theorem \ref{THM1}.  We need the following.
\begin{lem}\label{Lem2}{\rm (\cite{NST1})}
\emph{Let $(M,F)$ be a Finsler metric. Then $F$ is a GDW-metric if and only if
\begin{equation}\label{TP8}
D^i_{\ jkl|s}y^s=T_{jkl}y^i,
\end{equation}
for some tensor $T_{jkl}$ on manifold $M$.}
\end{lem}

\bigskip

\begin{prop}\label{Prop1}
Let $F$ be a non-Riemannian generalized isotropic Berwald metric. Then $F$ is a Douglas metric if and only if it is a relatively isotropic Landsberg metric ${\bf L}+F^2\lambda {\bf C}=0$.
\end{prop}
\begin{proof}
By assumption, we have
\be\label{B1}
B^i_{\ jkl}=\mu C_{jkl}\ell^i+\lambda(h^i_j h_{kl}+h^i_k h_{jl}+h^i_l h_{jk}),
\ee
Taking a trace of (\ref{B1}) yields
\be\label{B2}
E_{jk}=\frac{n+1}{2}\lambda h_{jk}.
\ee
Thus
\be\label{B3}
B^i_{\ jkl}=\mu C_{jkl}\ell^i+\frac{2}{n+1}(E_{jk}h^i_l+E_{kl}h^i_j+E_{jl}h^i_k ).
\ee
Contracting (\ref{B3}) with $y_i$ implies that
\be\label{B3b}
\mu  C_{jkl}=-2F^{-1}L_{jkl}.
\ee
By taking (\ref{B3b}) in (\ref{B3}) it follows that
\be\label{B3bb}
B^i_{\ jkl}=-2F^{-1}L_{jkl}\ell^i+\frac{2}{n+1}(E_{jk}h^i_l+E_{kl}h^i_j+E_{jl}h^i_k ).
\ee
On the other hand, we have
\begin{equation}\label{DB2}
h_{ij,k}=2C_{ijk}-F^{-2}(y_jh_{ik}+y_ih_{jk}),
\end{equation}
which implies that
\begin{equation}\label{DB3}
2E_{jk,l}=(n+1)\lambda_{,l} h_{jk}+(n+1)\lambda\Big\{2C_{jkl}-F^{-2}(y_kh_{jl}+y_jh_{kl})\Big\}.
\end{equation}
The Douglas tensor is given by
\begin{equation}\label{DB4}
D^i_{\ jkl}=B^i_{\ jkl}-\frac{2}{n+1}\{E_{jk}\delta^i_{\ l}+E_{kl}\delta^i_{\ j}+E_{lj}\delta^i_{\ k}+E_{jk,l}y^i\}.
\end{equation}
Putting (\ref{B2}), (\ref{B3bb}) and (\ref{DB3}) in (\ref{DB4}) yields
\begin{equation}\label{TP15}
D^i_{\ jkl}=-2\{F^{-2}L_{jkl}+\lambda C_{jkl}\}y^i-(\lambda y_lF^{-2}+\lambda_{,l})h_{jk} y^i.
\end{equation}
For the Douglas curvature, we have $D^i_{\ jkl}=D^i_{\ jlk}$. Then by (\ref{TP15}), we have
\begin{equation}\label{TP16}
\lambda y_lF^{-2}+\lambda_{,l}=0.
\end{equation}
From (\ref{TP15}) and (\ref{TP16}) we deduce that
\begin{equation}\label{TP17}
D^i_{\ jkl}=-2\{F^{-2}L_{jkl}+\lambda C_{jkl}\}y^i.
\end{equation}
By (\ref{TP17}), it follows that $F$ is a Douglas metric if and only if $F^{-2}L_{jkl}+\lambda C_{jkl}=0$. This completes the proof.
\end{proof}
\bigskip

\begin{cor}
Let $(M, F)$ be a non-Riemannian Finsler surface. Then $F$ is a Douglas metric if and only if $3I_{,1}+FII_2=0$.
\end{cor}
\begin{proof}
As we explain in introduction,  every two-dimensional Finsler manifolds are generalized isotropic Berwald  manifolds. By (\ref{B3b}) and  (\ref{TP17})  we get
\begin{equation}\label{2D}
D^i_{\ jkl}=\{F^{-1}\mu-2\lambda\} C_{jkl}y^i.
\end{equation}
Thus $F$ is a Douglas metric if and only if $\mu=2F\lambda$. Since $\mu=\frac{-2I_{,1}}{I}$ and  $\lambda=\frac{I_2}{3}$, then we get the proof.
\end{proof}

\bigskip

By (\ref{B2}), we have the following.

\begin{cor}
Let $(M, F)$ be a  Finsler surface. Then $F$ is a weakly Berwald metric if and only if $I_2=0$.
\end{cor}

\bigskip

\noindent {\bf Proof of Theorem \ref{THM1}}:  The Douglas tensor of $F$ is given by
\begin{equation}\label{TP19}
D^i_{\ jkl}=-2\{F^{-2}L_{jkl}+\lambda C_{jkl}\}y^i.
\end{equation}
Taking a horizontal derivation of (\ref{TP19}) implies that
\begin{equation}\label{TP20}
D^i_{\ jkl|s}y^s=-2\{F^{-2}L_{jkl|s}y^s+\lambda' C_{jkl}+\lambda L_{jkl}\}y^i.
\end{equation}
where $\lambda'=\lambda_{|m} y^m$. By Lemma \ref{Lem2}, $F$ is a GDW-metric with
\begin{equation}\label{TP21}
T_{jkl}=-2\{F^{-2}L_{jkl|s}y^s+\lambda' C_{jkl}+\lambda L_{jkl}\}.
\end{equation}
This completes the proof.
\qed
%------------------------------------------------------------------------------------------------------------
\section{Proof of Theorem \ref{THM2}}
%------------------------------------------------------------------------------------------------------------
In this section, we study complete generalized isotropic Berwald  manifold with vanishing stretch curvature.

\bigskip
\noindent
{\bf Proof of Theorem \ref{THM2}}: By  definition
\be
{\Sigma}_{ijkl}:=2(L_{ijk|l}-L_{ijl|k})=0.\label{St1}
\ee
Contracting (\ref{B0}) with $y_i$ and using
\[
y_iB^i_{\ jkl}=-2L_{jkl}
\]
implies that
\be
L_{ijk}=-\frac{1}{2}\mu FC_{ijk}.\label{St2}
\ee
By (\ref{St1}) and (\ref{St2}), we get
\be
\mu_{|k}C_{ijl}-\mu_{|l}C_{ijk}=\mu(C_{ijk|l}-C_{ijl|k}).\label{St3}
\ee
Contracting (\ref{St3}) with $y^k$ and using (\ref{St2}) yields
\be
(\mu'-\frac{1}{2}\mu^2 F)C_{ijk}=0.\label{St4}
\ee
If $C_{ijk}=0$, then $F$ is a Riemannian metric, and thus it is a Landsberg metric. Suppose that $F$ is a non-Riemannian metric. Then we have
\be
2\mu'=\mu^2 F.\label{St5}
\ee
On a Finslerian geodesics, we have
\be
\mu'=\mu'(t)=\frac{d \mu}{dt}.
\ee
Thus
\be
2\frac{d \mu}{dt}=\mu^2,
\ee
which its general solution is
\be
\mu(t)=\frac{2\mu(0)}{2-t\mu(0)}.
\ee
If $\mu(0)=0$, then $\mu(t)=0$ and by (\ref{St2}), we conclude that $F$ is a Landsberg metric. Suppose that $\mu(0)\neq 0$. Using $||\mu||<\infty$, and letting $t\rightarrow +\infty$ or $t\rightarrow -\infty$, implies  that $\mu=0$. This complete the proof.
\qed

%\bigskip By Theorem \ref{THM2}, we have the following.\begin{cor}Let $(M, F)$ be a  compact non-Riemannian generalized isotropic Berwald manifold. Then ${\bf \Sigma}=0$ if and only if $I_{,1}=0$.\end{cor}
%------------------------------------------------------------------------------------------------------------
\section{Proof of Theorem \ref{THM3}}
%------------------------------------------------------------------------------------------------------------
To prove Theorem \ref{THM3}, we need the following.
\begin{lem}\label{Lem1}{\rm (\cite{Ich}\cite{NST2})}
\emph{For the Berwald connection, the following Bianchi identities hold:}
\begin{eqnarray}
&&R^i_{\ jkl|m}+ R^i_{\ jlm|k}+R^i_{\ jmk|l}=B^i_{\ jku}R^u_{\ lm}+B^i_{\ jlu}R^u_{ mk}+B^i_{\ jmu}R^u_{\ kl}\label{TP22}\\
&&B^i_{\ jml|k}- B^i_{\ jmk|l}=R^i_{\ jkl,m}\label{TP23}\\
&&B^i_{\ jkl,m}=B^i_{\ jkm,l}\label{TP24}
\end{eqnarray}
where $R^i_{\ kl}:=\ell^j R^i_{\ jkl}$.
\end{lem}

Taking a trace of (\ref{TP23}) implies the following.
\begin{lem}{\rm (\cite{Mo1}\cite{TBN})}\label{lem2}
Let $F$ be a R-quadratic Finsler metric. Then\ \ ${\bf H}=0$.
\end{lem}
\bigskip

Contracting (\ref{TP23}) with $y_i$ yields
\begin{eqnarray}\label{eqq2}
   y_i R^i_{\,\,jkl,m}\!\!\!\!&=&\!\!\!\!\ y_iB^i_{\,\,jml|k}-y_iB^i_{\,\,jkm|l} \nonumber\\
    \!\!\!\!&=&\!\!\!\!\  (y_iB^i_{\,\,jml})_{|k}-(y_iB^i_{\,\,jkm})_{|l} \nonumber\\
  \!\!\!\!&=&\!\!\!\!\  -2L_{jml|k}+2L_{jkm|l}=\Sigma_{jkml}.
\end{eqnarray}
Thus we conclude the following.
\begin{cor}\label{cor1}
Every R-quadratic Finsler metric is a stretch metric.
\end{cor}

\bigskip

\noindent {\bf Proof of Theorem \ref{THM3}}: Let $(M, F)$ be a generalized isotropic Berwald manifold. Suppose that $F$ is  R-quadratic metric. By Corollary \ref{cor1}, it is sufficient to prove that ${\bf \bar E}=0$. By assumption, we have
\be\label{BR1}
B^i_{\ jkl}=-2F^{-1}L_{jkl}\ell^i+\frac{2}{n+1}(E_{jk}h^i_l+E_{kl}h^i_j+E_{jl}h^i_k ).
\ee
Then
\be\label{BR2}
B^i_{\ jkl|s}=-2F^{-1}L_{jkl|s}\ell^i+\frac{2}{n+1}(E_{jk|s}h^i_l+E_{kl|s}h^i_j+E_{jl|s}h^i_k ).
\ee
Replacing $l$ and $s$ in (\ref{BR2}) yields
\be\label{BR3}
B^i_{\ jks|l}=-2F^{-1}L_{jks|l}\ell^i+\frac{2}{n+1}(E_{jk|l}h^i_s+E_{ks|l}h^i_j+E_{js|l}h^i_k ).
\ee
(\ref{BR2})-(\ref{BR3}), implies that
\begin{eqnarray}\label{BR4}
\nonumber B^i_{\ jkl|s}-B^i_{\ jks|l}=\!\!\!\!&-&\!\!\!\!\ 2F^{-1}\{L_{jkl|s}-L_{jks|l}\}\ell^i+\frac{2}{n+1}(E_{jk|s}h^i_l-E_{jk|l}h^i_s)\\
\!\!\!\!&+&\!\!\!\!\ \frac{2}{n+1}\Big\{(E_{jl|s}-E_{js|l})h^i_k+(E_{kl|s}-E_{ks|l})h^i_j\Big\}.
\end{eqnarray}
By (\ref{TP23}), (\ref{BR4}) and Corollary \ref{cor1}, we have
\be\label{BR5}
E_{jk|l}h^i_s-E_{jk|s}h^i_l=(E_{jl|s}-E_{js|l})h^i_k+(E_{kl|s}-E_{ks|l})h^i_j.
\ee
Putting $i=s$ in  (\ref{BR5}) and using $h^s_s=n-1$ and $h^s_l=\delta^s_l-F^{-2}y^sy_l$, we get
\[
(n-2)E_{jk|l}+F^{-2}H_{jk}y_l=(E_{jl|k}-F^{-2}H_{jl}y_k-E_{jk|l})+(E_{kl|j}-F^{-2}H_{kl}y_j-E_{jk|l}),
\]
or equivalently
\be \label{BR6}
nE_{jk|l}=E_{jl|k}+E_{kl|j}-F^{-2}(H_{jl}y_k-H_{kl}y_j+H_{jk}y_l).
\ee
By Lemma \ref{lem2}, (\ref{BR6}) reduces to following
\be \label{BR7}
nE_{jk|l}=E_{jl|k}+E_{kl|j}.
\ee
Permuting $j, k,l$ in (\ref{BR7}) leads to
\begin{eqnarray}
nE_{kl|j}\!\!\!\!&=&\!\!\!\!\ E_{kj|l}+E_{lj|k}\label{BR8}\\
nE_{lj|k}\!\!\!\!&=&\!\!\!\!\ E_{lk|j}+E_{jk|l}.\label{BR9}
\end{eqnarray}
(\ref{BR7})+(\ref{BR8})-(\ref{BR9}) yields
\be \label{BR10}
n(E_{jk|l}+E_{kl|j})=(n+2)E_{jl|k}.
\ee
Putting (\ref{BR7}) in  (\ref{BR10}) implies that
\be \label{BR11}
E_{kl|j}=E_{jl|k}.
\ee
This means that ${\bar E}_{ijk}$ is symmetric with respect to indices  and (\ref{BR7}) reduces to
\[
(n-2)E_{jk|l}=0.
\]
Since $n>2$, thus  ${\bf \bar E}=0$.

Conversely, let $F$ be a stretch metric with \  ${\bf \bar E}=0$. Then by (\ref{TP23}) and (\ref{BR4}), we conclude that $F$ is R-quadratic. This completes the proof.
\qed

%------------------------------------------------------------------------------------------------------------
\section{Proof of Theorem \ref{THM4}}
%------------------------------------------------------------------------------------------------------------
The following equation is hold
\begin{eqnarray}
L_{ijk|m}y^m + C_{ijm}R^m_{\ k}= - {1\over 3}(g_{im}R^m_{\ k, j}
+ g_{jm} R^m_{\ k, i}) - {1\over 6}( g_{im}R^m_{\ j, k}+g_{jm}R^m_{\ i, k}). \label{Moeq1}
\end{eqnarray}
For more details, see  \cite{MoSh}. Contracting (\ref{Moeq1}) with $g^{ij}$ implies that
\be
J_{k|m}y^m + I_mR^m_{\ k}  = -  {1\over 3}\Big \{ 2 R^m_{\ k, m} +  R^m_{\ m, k}\Big \} . \label{Moeq2}
\ee

\bigskip
\noindent
{\bf Proof of Theorem \ref{THM4}}:  Now we assume that $F$ is of scalar curvature with flag curvature ${\bf K}= {\bf K}(x, y)$. This is equivalent to the following identity:
\be
R^i_{\ k} = {\bf K} F^2 \; h^i_k, \label{Kikiso1}
\ee
where $h^i_k := g^{ij} h_{jk}$. By  (\ref{Moeq1}),  (\ref{Moeq2}) and (\ref{Kikiso1}), we obtain
\[
 L_{ijk|m}y^m =  - {1\over 3}F^2 \Big \{ {\bf K}_{, i} h_{jk}   + {\bf K}_{, j} h_{ik}  + {\bf K}_{, k} h_{ij} + 3 {\bf K} C_{ijk}\Big \}
\]
and
\be
J_{k|m}y^m = - {1\over 3}F^2\Big \{ (n+1) {\bf K}_{, k} + 3{\bf K} I_k \Big \}.\label{AZeq2}
\ee
By assumption, we have
\be
L_{jkl}=-\frac{1}{2}\mu FC_{jkl}.
\ee
This yields
\[
J_i=-\frac{1}{2}\mu FI_i.
\]
Since $J_k = I_{k|m}y^m$, thus
\be
J_{i|m}y^m=-\frac{\mu'}{2}FI_i-\frac{\mu}{2}FJ_i=\frac{1}{4}(\mu^2F-2\mu')FI_i.
\ee
It follows from (\ref{AZeq2}) that
\be
{n+1 \over 3} {\bf K}_{,i}=(\frac{\mu'}{2F}-\frac{\mu^2}{4}-{\bf K})I_i. \label{KKcI}
\ee
Then  we have (\ref{KKc}).
\qed

\bigskip

\begin{cor}
Let $(M, F)$ be a generalized isotropic Berwald manifold of dimension $n\geq 3$. Suppose that $F$ is of scalar flag curvature ${\bf K}={\bf K}(x, y)$ such that $2\mu'-\{\mu^2+4{\bf K}\}F=0$. Then $F$ is of constant flag curvature.
\end{cor}
\begin{proof}
By (\ref{KKcI}), we get ${\bf K}_{,i}=0$  and then ${\bf K}={\bf K}(x)$.  In this case, ${\bf K}= constant$ when $n \geq 3$ by the Schur theorem \cite{BCS}.
\end{proof}

\bigskip

Finally,  by  (\ref{KKcI}) we can conclude the following.

\begin{cor}
Let $(M, F)$ be a generalized isotropic Berwald manifold of isotropic  flag curvature ${\bf K}={\bf K}(x)$  satisfies $2\mu'-\{\mu^2+4{\bf K}\}F\neq0$.  Then $F$ reduces to a  Riemannian metric.
\end{cor}

%-----------------------------------------------------------------------------------------------------------------------

\noindent
Akbar Tayebi\\
Department of Mathematics, Faculty  of Science\\
Qom University\\
Qom. Iran\\
Email:\ akbar.tayebi@gmail.com

\bigskip

\noindent
Esmaeil Peyghan\\
Department of Mathematics, Faculty  of Science\\
Arak University\\
Arak 38156-8-8349,  Iran\\
Email: epeyghan@gmail.com

\end{document}